\documentclass[11pt, a4paper]{article}
\usepackage{indentfirst}
\usepackage{amsfonts}
\usepackage{amssymb}
\usepackage{mathrsfs}
\usepackage{amsmath}
\usepackage{amsthm}
\usepackage{enumerate}
\usepackage{cite}
\usepackage{mathrsfs}
\usepackage{longtable}
\allowdisplaybreaks
\usepackage{geometry}
\usepackage{ifpdf}
\ifpdf
\usepackage[colorlinks=true, linkcolor=blue, citecolor=red, final, backref=page, hyperindex]{hyperref}
\else
\usepackage[colorlinks, final, backref=page, hyperindex, hypertex]{hyperref}
\fi
\usepackage{longtable}
\usepackage{txfonts}
\usepackage{indentfirst}
\usepackage{latexsym}
\usepackage[all]{xy}

\voffset=- 1.5cm
\def\<{\langle}
\def\>{\rangle}

\newtheorem{thm}{Theorem}[section]

\newtheorem{prop}[thm]{Proposition}

\theoremstyle{definition}
\newtheorem{defn}{Definition}[section]

\theoremstyle{remark}
\newtheorem{re}{Remark}[section]
\begin{document}
	\title{\bf Classification and (Quasi)-Centroids of Four-Dimensional Ternary Leibniz Algebras}
	\author{\bf  Ahmed Zahari Abdou Damdji}
	\author{{
	\ Ahmed Zahari Abdou Damdji$^{1}$
			\footnote { Corresponding author, E-mail: abdou-damdji.ahmed-zahari@uha.fr}.
		}\\
		\\
		{\small  IRIMAS-Department of Mathematics, Faculty of Sciences, University of Haute Alsace, Mulhouse, France.}
	}
	\date{}
	\maketitle
	
	\begin{abstract}
We provide a  classification, up to isomorphism, of four-dimensional ternary Leibniz algebras over an algebraically closed field of characteristic zero. 
For each non-abelian algebra in the classification, we explicitly determine its centroid and quasi-centroid and compute their dimensions. 
These results offer a comprehensive description of the internal symmetries of low-dimensional ternary Leibniz algebras and extend several classical 
results from the binary Leibniz setting to the ternary case.
\end{abstract}

	\textbf{Key words:}\ Ternary Leibniz algebra, classification, Centroid, quasi-centroid. \\

	\numberwithin{equation}{section}
	
	\tableofcontents
\section{Introduction}

Leibniz algebras were introduced by J.-L.~Loday \cite{Loday1993} as a non-antisymmetric generalization of Lie algebras, preserving the Leibniz identity as a weakened
 form of the Jacobi identity. Since their introduction, Leibniz algebras have attracted considerable attention due to their rich algebraic structure and their close 
connections with homological algebra, geometry, and theoretical physics see\cite{Loday1996, Casas2012}. In recent years, increasing 
interest has been devoted to $n$-ary generalizations of classical algebraic structures, motivated both by abstract algebraic considerations and by applications in 
mathematical physics, particularly in the study of Nambu mechanics and higher-order symmetries \cite{Filippov1985, Takhtajan1994a, Takhtajan1994b}.

Among these generalizations, ternary Leibniz algebras arise naturally as a non-skew-symmetric counterpart of ternary (or $3$-Lie) algebras. A ternary Leibniz algebra
 is a vector space equipped with a trilinear operation satisfying a Leibniz-type identity that generalizes the derivation property of binary Leibniz algebras.
 These structures extend the class of ternary Lie algebras and provide a broader framework in which symmetry assumptions are relaxed, allowing for a richer variety
 of algebraic behaviors. Ternary and, more generally, $n$-ary Leibniz-type structures have been investigated from both algebraic and cohomological perspectives 
\cite{Makhlouf2020, ZahariBakayoko2023, Casas2014, Guo2016}.

The classification problem plays a fundamental role in understanding algebraic structures of low dimension. Low-dimensional classifications not only provide explicit 
examples and counterexamples, but also serve as a testing ground for conjectures and structural properties in higher dimensions. In the context of ternary Leibniz 
algebras, complete classifications remain relatively limited, particularly in dimensions greater than three. Therefore, the classification of $4$-dimensional ternary 
Leibniz algebras represents a natural and meaningful step toward a deeper understanding of these algebras and their structural features.

Another important aspect in the study of algebraic structures is the investigation of their centroids and quasi-centroids. The centroid of an algebra consists of
 linear maps that commute with the algebra multiplication, while the quasi-centroid generalizes this notion by relaxing certain commutativity conditions. 
These objects provide valuable insight into the internal symmetries of the algebra and are closely related to derivations, automorphisms, and deformation 
theory \cite{Leger2000, Albeverio2011, ZahariSania2025}. For Leibniz and $n$-ary algebras, quasi-centroids play a crucial role in describing invariant operators, structural decompositions, 
and symmetry-related properties \cite{Sheng2016, Mosbahi2023, Zahari2019, Beatrice2025, Mosbahi2025, ZahariBouzid2024}.

The aim of this paper is twofold. First, we provide a complete classification, up to isomorphism, of $4$-dimensional ternary Leibniz algebras over an algebraically 
closed field of characteristic zero. Second, for each algebra appearing in the classification, we explicitly determine its quasi-centroid, highlighting how these structures
 reflect the underlying algebraic properties. Our results contribute to the growing body of work on $n$-ary Leibniz algebras and offer a foundation for further studies on 
derivations, deformations, and cohomology in higher-dimensional and higher-arity settings.

The paper is organized as follows. Section~2 is devoted to the classification of $4$-dimensional ternary Leibniz algebras. 
 In Section~3, we recall basic definitions and preliminary results concerning ternary Leibniz algebras and quasi-centroids.
In Section~4, we compute and analyze the quasi-centroids of the classified algebras.  Finally, concluding remarks and perspectives for future research are presented.

\section{Classifications of ternary Leibniz algebras}

In this section, we give the classification of Ternary Leibniz algebra in low dimension.

 For a selection of cases, we classify the ternary Leibniz algebras. The computations for our classifications were done using Mathematica. 
Throughout, we work over the complex field.	

\begin{defn}\label{dt}
A \emph{ternary Leibniz algebra} is a vector space $\mathcal{L}$ over a field $\mathbb{K}$ endowed with a trilinear map
$\llbracket -,-,- \rrbracket : \mathcal{L}^{\otimes 3} \to \mathcal{L}$
satisfying the Leibniz identity
\[
\llbracket \llbracket x,y,z \rrbracket , t , u \rrbracket
=
\llbracket x,y,\llbracket z,t,u \rrbracket \rrbracket
+
\llbracket x,\llbracket y,t,u \rrbracket , z \rrbracket
+
\llbracket \llbracket x,t,u \rrbracket , y , z \rrbracket
\]
for all $x,y,z,t,u \in \mathcal{L}$.
\end{defn}

Let $(\mathcal{L}, \llbracket-,-,-\rrbracket)$ be an $n$-dimensional Ternary Leibniz algebra, $\{e_i\}$ be a basis of $\mathcal{L}$. For any $i, j\in \mathbb{N},
 1\leq i, j\leq n$, let us put 
$$\llbracket e_i, e_j, e_k\rrbracket=\sum_{p=1}^{n}\chi_{ijk}^pe_p$$
The axioms in Definition (\ref{dt}) are respectively equivalent to
\begin{eqnarray}
\sum_{r=1}^n\Bigg(\chi_{ijk}^r\chi_{rpq}^s-\chi_{kpq}^r\chi_{ijr}^s-\chi_{jpq}^r\chi_{irk}^s-\chi_{ipq}^r\chi_{rjk}^s\Bigg)&=&0.
\end{eqnarray}

\begin{thm}
The isomorphism class of 4-dimensional ternary Leibniz algebras given by  the following representatives.
\begin{itemize}
\item
$\mathcal{L}_1$ :
 $\begin{array}{ll}  
\llbracket e_1,e_1, e_1\rrbracket=e_2+e_4,\quad \llbracket e_1,e_1, e_3\rrbracket=e_2+e_4, \,\llbracket e_1,e_3, e_1\rrbracket=e_2+e_4,\,
\llbracket e_3,e_1, e_1\rrbracket=e_2+e_4,\\ \llbracket e_3,e_1, e_3\rrbracket=e_2+e_4,\quad \llbracket e_3,e_3, e_3\rrbracket=e_2+e_4;
\end{array}$
\item
$\mathcal{L}_2$ :
 $\begin{array}{ll}  
\llbracket e_1,e_1, e_2\rrbracket=e_3+e_4,\quad \llbracket e_1,e_2, e_1\rrbracket=2e_3+e_4, \,\llbracket e_2,e_1, e_1\rrbracket=e_3+e_4,\,
\llbracket e_2,e_1, e_2\rrbracket=e_3+e_4,\\ \llbracket e_2,e_2, e_2\rrbracket=e_3+e_4;
\end{array}$

\item
$\mathcal{L}_3$ :
 $\begin{array}{ll}  
\llbracket e_1,e_1, e_4\rrbracket=e_3,\ \llbracket e_1,e_2, e_2\rrbracket=e_3, \,\llbracket e_1,e_2, e_4\rrbracket=e_3,\,
\llbracket e_1,e_4, e_4\rrbracket=e_3,\,\llbracket e_2,e_1, e_1\rrbracket=e_3,\\
\llbracket e_2,e_2, e_1\rrbracket=e_3,\ \llbracket e_2,e_2, e_4\rrbracket=e_3, \,\llbracket e_2,e_4, e_4\rrbracket=e_3,\,
\llbracket e_4,e_1, e_1\rrbracket=e_3,\,\llbracket e_4,e_2, e_2\rrbracket=e_3,\\
\llbracket e_4,e_4, e_1\rrbracket=e_3,\quad\llbracket e_4,e_4, e_4\rrbracket=e_3;
\end{array}$

\item
$\mathcal{L}_4$ :
 $\begin{array}{ll}  
\llbracket e_1,e_2, e_2\rrbracket=e_1+e_3,\,\llbracket e_1,e_4, e_4\rrbracket=e_1+e_3,\,\llbracket e_2,e_2, e_2\rrbracket=e_3,\,
\llbracket e_2,e_2, e_4\rrbracket=e_3,\,\llbracket e_2,e_4, e_2\rrbracket=e_3,\\
\llbracket e_2,e_4, e_4\rrbracket=e_3,\, \llbracket e_4,e_2, e_2\rrbracket=e_3, \,\llbracket e_4,e_2, e_4\rrbracket=e_3,\,
\llbracket e_4,e_4, e_2\rrbracket=e_3,\,\llbracket e_4,e_4, e_4\rrbracket=e_3;
\end{array}$

\item
$\mathcal{L}_5$ :
 $\begin{array}{ll}  
\llbracket e_1,e_2, e_3\rrbracket=e_4,\,  \llbracket e_1,e_3, e_1\rrbracket=e_4,\,\llbracket e_1,e_3, e_2\rrbracket=e_4,\, 
\llbracket e_1,e_3, e_3\rrbracket=e_4,\, \llbracket e_2,e_1, e_3\rrbracket=e_4,\\
\llbracket e_2,e_3, e_1\rrbracket=e_4,\,\llbracket e_3,e_1, e_3\rrbracket=e_4, \, \llbracket e_3,e_2, e_1\rrbracket=e_4,\, 
\llbracket e_3,e_3, e_1\rrbracket=e_4;
\end{array}$

\item
$\mathcal{L}_6$ :
 $\begin{array}{ll}  
\llbracket e_1,e_2, e_4\rrbracket=e_1+e_3,\,\llbracket e_1,e_4, e_2\rrbracket=e_1+e_3,\,\, \llbracket e_2,e_2, e_2\rrbracket=e_3,\,
\llbracket  e_2,e_2, e_4\rrbracket=e_3,\\\ \llbracket  e_2,e_4, e_2\rrbracket=e_3,\,\llbracket e_4,e_2, e_2\rrbracket=\frac{1}{2}e_3,\,
\llbracket e_4,e_4, e_2\rrbracket=e_3,\quad \llbracket e_4,e_4, e_4\rrbracket=e_3;
\end{array}$

\item
$\mathcal{L}_7^\alpha$ :
 $\begin{array}{ll}  
\llbracket e_1,e_3, e_2\rrbracket=2\alpha e_1+e_4,\, \llbracket e_2,e_2, e_2\rrbracket=e_4,\, \llbracket e_2,e_3, e_2\rrbracket=e_4,\,
\llbracket  e_2,e_3, e_3\rrbracket=e_4,\\ \llbracket e_3,e_2, e_2\rrbracket=e_4,\,\llbracket e_3,e_2,e_3\rrbracket=e_4,\quad
 \llbracket e_3,e_3, e_2\rrbracket=e_4,\quad \llbracket e_3,e_3, e_3\rrbracket=\frac{1}{3}e_4;
\end{array}$

\item
$\mathcal{L}_8$ :
 $\begin{array}{ll}  
\llbracket e_1,e_1, e_1\rrbracket=e_2,\,\llbracket e_1,e_3, e_3\rrbracket=e_2,\,\llbracket e_3,e_3, e_1\rrbracket=e_2,\,
\llbracket  e_3,e_3, e_3\rrbracket=e_2,\quad \llbracket e_4,e_1, e_1\rrbracket=e_2,\\
\llbracket e_4,e_4,e_4\rrbracket=e_2;
\end{array}$

\item
$\mathcal{L}_9$ :
 $\begin{array}{ll}  
\llbracket e_1,e_1, e_1\rrbracket=e_2,\,\llbracket e_1,e_4, e_1\rrbracket=e_2,\,\llbracket e_1,e_4, e_4\rrbracket=e_2,\,\,
\llbracket  e_3,e_3, e_1\rrbracket=e_2,\, \llbracket e_4,e_1, e_1\rrbracket=e_2,\\
\llbracket e_4,e_3,e_3\rrbracket=e_2,\quad \llbracket e_4,e_4,e_4\rrbracket=e_2;
\end{array}$

\item
$\mathcal{L}_{10}$ :
 $\begin{array}{ll}  
\llbracket e_2,e_1, e_1\rrbracket=e_2,\, \llbracket e_2,e_3, e_2\rrbracket=e_2,\,\llbracket e_2,e_3, e_4\rrbracket=e_2,\,
\llbracket  e_2,e_4, e_3\rrbracket=e_2,\, \llbracket e_2,e_4, e_4\rrbracket=e_2;
\end{array}$

\item
$\mathcal{L}_{11}$ :
 $\begin{array}{ll}  
\llbracket e_2,e_1, e_1\rrbracket=e_4,\,\llbracket e_2,e_1, e_2\rrbracket=e_4,\,\llbracket e_2,e_2, e_1\rrbracket=e_4,\,
\llbracket  e_2,e_2, e_2\rrbracket=e_4,\, \llbracket e_3,e_3, e_3\rrbracket=e_4;
\end{array}$

\item
$\mathcal{L}_{12}$ :
 $\begin{array}{ll}  
\llbracket e_2,e_1, e_2\rrbracket=e_3,\, \llbracket e_2,e_2, e_1\rrbracket=e_3,\,\llbracket e_2,e_2, e_2\rrbracket=e_3,\,
\llbracket  e_2,e_2, e_4\rrbracket=e_3,\,\llbracket e_2,e_4, e_2\rrbracket=e_3,\\
\llbracket e_4,e_1, e_1\rrbracket=e_3,\quad\llbracket  e_4,e_1, e_4\rrbracket=e_3,\quad \llbracket  e_4,e_2, e_2\rrbracket=e_3,
\quad \llbracket  e_4,e_4, e_4\rrbracket=e_3;
\end{array}$

\item
$\mathcal{L}_{13}$ :
 $\begin{array}{ll}  
\llbracket e_2,e_1, e_3\rrbracket=e_2,\,\llbracket e_2,e_3, e_1\rrbracket=e_2,\,\llbracket e_2,e_3, e_4\rrbracket=e_2,\,
\llbracket  e_2,e_4, e_3\rrbracket=e_2,\quad \llbracket e_2,e_4, e_4\rrbracket=e_2;
\end{array}$

\item
$\mathcal{L}_{14}$ :
 $\begin{array}{ll}  
\llbracket e_2,e_1, e_3\rrbracket=e_4,\,\llbracket e_2,e_2, e_2\rrbracket=e_4,\,\llbracket e_2,e_3, e_1\rrbracket=e_4,\,
\llbracket  e_3,e_1, e_1\rrbracket=e_4,\,\llbracket e_3,e_1, e_2\rrbracket=e_4,\\
\llbracket e_3,e_2, e_1\rrbracket=e_4,\quad\llbracket  e_3,e_3, e_3\rrbracket=e_4;
\end{array}$

\item
$\mathcal{L}_{15}$ :
 $\begin{array}{ll}  
\llbracket e_2,e_3, e_1\rrbracket=e_2,\quad \llbracket e_2,e_3, e_3\rrbracket=e_2,\quad\llbracket e_2,e_4, e_3\rrbracket=e_2,\quad\llbracket  e_2,e_4, e_4\rrbracket=e_2.
\end{array}$

\item
$\mathcal{L}_{16}$ :
 $\begin{array}{ll}  
\llbracket e_2,e_2, e_1\rrbracket=e_4,\quad \llbracket e_2,e_2, e_2\rrbracket=e_4,\quad\llbracket e_3,e_2, e_2\rrbracket=e_4,\quad
\llbracket  e_3,e_3, e_3\rrbracket=e_4;
\end{array}$

\item
$\mathcal{L}_{17}$ :
 $\begin{array}{ll}  
\llbracket e_2,e_2, e_3\rrbracket=e_4,\quad \llbracket e_3,e_2, e_1\rrbracket=e_4,\quad\llbracket e_3,e_2, e_2\rrbracket=e_4,\quad\llbracket  e_3,e_3, e_3\rrbracket=e_4;
\end{array}$

\item
$\mathcal{L}_{18}^\alpha$ :
 $\begin{array}{ll}  
\llbracket e_2,e_1, e_4\rrbracket=e_3,\,\llbracket e_2,e_2, e_2\rrbracket=e_3,\,\llbracket e_2,e_4, e_1\rrbracket=e_3,\,
\llbracket  e_4,e_1, e_2\rrbracket=-e_3,\, \llbracket e_4,e_2, e_1\rrbracket=e_3,\\
\llbracket e_4,e_2, e_2\rrbracket=\alpha e_3,\quad\llbracket  e_4,e_4, e_4\rrbracket=e_3.
\end{array}$
\end{itemize}
\end{thm}

\section{Quasi-Centroids of Ternary Leibniz Algebras}

\begin{defn}
Let $(\mathcal{L},\llbracket-,-,-\rrbracket)$ be a ternary Leibniz algebra over a field $\mathbb{K}$.
The centroid of $\mathcal{L}$ is defined as
$$
Cent(\mathcal{L})=\left\{\Phi \in End_{\mathbb{K}}(\mathcal{L}) \ \middle| \, \Phi(\llbracket x,y,z\rrbracket)
=\llbracket \Phi(x),y,z\rrbracket=\llbracket x,\Phi(y),z\rrbracket=\llbracket x,y,\Phi(z)\rrbracket,
\ \forall x,y,z \in \mathcal{L}
\right\}.
$$
\end{defn}
We denote the centroid of $\mathcal{L}$ by $\Gamma(\mathcal{L})$.
\begin{prop}
Let $(\mathcal{L},\llbracket-,-,-\rrbracket)$ be a ternary Leibniz algebra.
Then $Cent(\mathcal{L})$ is a subalgebra of $End_{\mathbb{K}}(\mathcal{L})$, that is,
it is closed under addition, scalar multiplication, and composition.
\end{prop}

\begin{proof}
Let $\Phi_1,\Phi_2 \in Cent(\mathcal{L})$ and $\alpha,\beta \in \mathbb{K}$.
For all $x,y,z \in \mathcal{L}$, by linearity we have
$$
(\alpha\Phi_1+\beta\Phi_2)(\llbracket x,y,z\rrbracket)=\alpha \llbracket \Phi_1(x),y,z\rrbracket
+\beta \llbracket \Phi_2(x),y,z\rrbracket=\llbracket (\alpha\Phi_1+\beta\Phi_2)(x),y,z\rrbracket.
$$
The same equality holds when $\alpha\Phi_1+\beta\Phi_2$ acts on the second or third argument, hence $\alpha\Phi_1+\beta\Phi_2 \in Cent(\mathcal{L})$.

Moreover,
$$
(\Phi_1 \circ \Phi_2)(\llbracket x,y,z\rrbracket)=\Phi_1(\llbracket \Phi_2(x),y,z\rrbracket)
=\llbracket (\Phi_1 \circ \Phi_2)(x),y,z\rrbracket,
$$
and similarly for the other arguments.
Thus $\Phi_1 \circ \Phi_2 \in Cent(\mathcal{L})$.
\end{proof}

\begin{defn}\label{qc}
Let $(\mathcal{L},\llbracket-,-,-\rrbracket)$ be a ternary Leibniz algebra over $\mathbb{K}$.
The quasi-centroid of $\mathcal{L}$ is defined by
$$
QCent(\mathcal{L})=\left\{\Phi \in End_{\mathbb{K}}(\mathcal{L}) \ \middle| \
\llbracket \Phi(x),y,z\rrbracket
=\llbracket x,\Phi(y),z\rrbracket=\llbracket x,y,\Phi(z)\rrbracket,
\ \forall x,y,z \in \mathcal{L}
\right\}.
$$
\end{defn}

\begin{prop}\label{pc1}
Let $(\mathcal{L},\llbracket-,-,-\rrbracket)$ be a ternary Leibniz algebra over $\mathbb{K}$.
Then:
\begin{enumerate}
\item[(i)] $QCent(\mathcal{L})$ is a vector subspace of $End_{\mathbb{K}}(\mathcal{L})$.
\item[(ii)] $Cent(\mathcal{L}) \subseteq QCent(\mathcal{L})$.
\item[(iii)] If $\mathcal{L}$ is a commutative ternary algebra, then
$QCent(\mathcal{L}) = Cent(\mathcal{L})$.
\item[(iv)] If the ternary product is associative, then $QCent(\mathcal{L})$ is closed under composition.
\end{enumerate}
\end{prop}

\begin{proof}
(i) follows immediately from linearity of the ternary product.

(ii) Let $\Phi \in Cent(\mathcal{L})$.
By definition,
$$
\Phi(\llbracket x,y,z\rrbracket)=\llbracket \Phi(x),y,z\rrbracket
=\llbracket x,\Phi(y),z\rrbracket=\llbracket x,y,\Phi(z)\rrbracket,
$$
which implies $\Phi \in QCent(\mathcal{L})$.

(iii) Assume that $\mathcal{L}$ is commutative and let $\Phi \in QCent(\mathcal{L})$.
Then
$$
\Phi(\llbracket x,y,z\rrbracket)=\llbracket \Phi(x),y,z\rrbracket,
$$
and by symmetry of the product, the centroid identities follow.
Hence $\Phi \in Cent(\mathcal{L})$, and the result follows from (ii).

(iv) Let $\Phi_1,\Phi_2 \in QCent(\mathcal{L})$.
Using associativity of the ternary product, we obtain
$$
\llbracket (\Phi_1\circ\Phi_2)(x),y,z\rrbracket=\llbracket x,(\Phi_1\circ\Phi_2)(y),z\rrbracket
=\llbracket x,y,(\Phi_1\circ\Phi_2)(z)\rrbracket,
$$
which shows that $\Phi_1\circ\Phi_2 \in QCent(\mathcal{L})$.
\end{proof}

\begin{defn}
Let $(\mathcal{L},\llbracket -,-,- \rrbracket)$ be a ternary Leibniz algebra over a field $\mathbb{K}$.
A linear map $D \in End_{\mathbb{K}}(\mathcal{L})$ is called a derivation of $\mathcal{L}$ if it satisfies the Leibniz rule
$$
D(\llbracket x,y,z \rrbracket)=\llbracket D(x),y,z \rrbracket
+\llbracket x,D(y),z \rrbracket+\llbracket x,y,D(z) \rrbracket
$$
for all $x,y,z \in \mathcal{L}$.
The set of all derivations of $\mathcal{L}$ is denoted by $Der(\mathcal{L})$.
\end{defn}

\begin{defn}
Let $(\mathcal{L},\llbracket-,-,-\rrbracket)$ be a ternary Leibniz algebra.
A linear map $\psi \in End_{\mathbb{K}}(\mathcal{L})$ is called a
central derivation if
$$
\psi(\llbracket x,y,z\rrbracket)=\llbracket \psi(x),y,z\rrbracket
=\llbracket x,\psi(y),z\rrbracket=\llbracket x,y,\psi(z)\rrbracket
=0,
\quad \forall x,y,z \in \mathcal{L}.
$$
The set of all central derivations of $\mathcal{L}$ is denoted by $C(\mathcal{L})$.
\end{defn}

\begin{prop}
Let $(\mathcal{L},\llbracket-,-,-\rrbracket)$ be a ternary Leibniz algebra over a field $\mathbb{K}$.
Then
$$
C(\mathcal{L})=\Gamma(\mathcal{L})\cap \mathrm{Der}(\mathcal{L}).
$$
\end{prop}

\begin{proof}
Let $\psi\in \Gamma(\mathcal{L})\cap \mathrm{Der}(\mathcal{L})$.
Since $\psi$ is a derivation, for all $x,y,z\in\mathcal{L}$ we have
\begin{equation}\label{eq1}
\psi(\llbracket x,y,z\rrbracket)
=
\llbracket\psi(x),y,z\rrbracket+\llbracket x,\psi(y),z\rrbracket+\llbracket x,y,\psi(z)\rrbracket.
\end{equation}
On the other hand, since $\psi$ belongs to the centroid, we also have
\begin{equation}\label{eq2}
\psi(\llbracket x,y,z\rrbracket)=\llbracket\psi(x),y,z\rrbracket=\llbracket x,\psi(y),z\rrbracket=\llbracket x,y,\psi(z)\rrbracket.
\end{equation}
Comparing \eqref{eq1} and \eqref{eq2}, it follows that
$$
\llbracket\psi(x),y,z\rrbracket=\llbracket x,\psi(y),z\rrbracket=\llbracket x,y,\psi(z)\rrbracket=0,
\quad \forall x,y,z\in\mathcal{L}.
$$
Hence $\psi(\llbracket x,y,z\rrbracket)=0$ for all $x,y,z\in\mathcal{L}$, which proves that
$\psi$ is a central derivation. Therefore,
$$
\Gamma(\mathcal{L})\cap \mathrm{Der}(\mathcal{L}) \subseteq C(\mathcal{L}).
$$
Conversely, let $\psi\in C(\mathcal{L})$. Then $\psi$ is a derivation satisfying
$\psi(\llbracket x,y,z\rrbracket)=0$ for all $x,y,z\in\mathcal{L}$. Consequently,
$$
\llbracket\psi(x),y,z\rrbracket=\llbracket x,\psi(y),z\rrbracket=\llbracket x,y,\psi(z)\rrbracket=0,
$$
which implies that $\psi$ commutes with the ternary product in each argument.
Thus $\psi\in\Gamma(\mathcal{L})$, and hence
$$
C(\mathcal{L}) \subseteq \Gamma(\mathcal{L})\cap \mathrm{Der}(\mathcal{L}).
$$

Combining both inclusions, we conclude that
$$
C(\mathcal{L})=\Gamma(\mathcal{L})\cap \mathrm{Der}(\mathcal{L}).
$$
\end{proof}


\begin{prop}
Let $(\mathcal{L},\llbracket-,-,-\rrbracket)$ be a ternary Leibniz algebra. Then
\begin{enumerate}
    \item[(i)] $\Gamma(\mathcal{L}) \, Der(\mathcal{L}) \subseteq Der(\mathcal{L})$.
    \item[(ii)] $\left[\Gamma(\mathcal{L}), Der(\mathcal{L})\right]\subseteq \Gamma(\mathcal{L})$.
    \item[(iii)] $\left[\Gamma(\mathcal{L}), \Gamma(\mathcal{L})\right](\mathcal{L}) \subseteq \Gamma(\mathcal{L})$ and $\left[\Gamma(\mathcal{L}), \Gamma(\mathcal{L})\right](\mathcal{L}^3) = 0$.
\end{enumerate}
\end{prop}

\begin{proof}
\begin{enumerate}
    \item[(i)] Let $\varphi \in \Gamma(\mathcal{L})$ and $\kappa \in Der(\mathcal{L})$. Then for all $x,y,z \in \mathcal{L}$,
    $$
    (\varphi \circ \kappa)(\llbracket x,y,z\rrbracket) = \varphi(\kappa(\llbracket x,y,z\rrbracket)).
    $$
    Using the derivation property of $\kappa$ and the centroid property of $\varphi$, we obtain
   $$
    (\varphi \circ \kappa)(\llbracket x,y,z\rrbracket) =\llbracket(\varphi \circ \kappa)(x),y,z\rrbracket +\llbracket x,(\varphi \circ \kappa)(y),z\rrbracket +\llbracket x,y,(\varphi \circ \kappa)(z)\rrbracket,
    $$
    showing that $\varphi\circ \kappa \in Der(\mathcal{L})$.

    \item[(ii)] Let $\varphi \in \Gamma(\mathcal{L})$ and $\kappa \in Der(\mathcal{L})$. For all $x,y,z\in \mathcal{L}$,
    \[
    [\varphi,\kappa](\llbracket x,y,z\rrbracket) = (\varphi\circ\kappa - \kappa\circ\varphi)(\llbracket x,y,z\rrbracket).
    \]
    Using the definitions of derivation and centroid, we see that
   $$
  \left[\varphi,\kappa\right](\llbracket x,y,z\rrbracket) =\llbracket\left[\varphi,\kappa\right](x),y,z\rrbracket = \llbracket x,\left[\varphi,\kappa\right](y),z\rrbracket = \llbracket x,y,\left[\varphi,\kappa\right](z)\rrbracket,
   $$
    which shows $\llbracket\varphi,\kappa\rrbracket\in \Gamma(\mathcal{L})$.

    \item[(iii)] Let $\varphi, \psi \in \Gamma(\mathcal{L})$. Then, for all $x,y,z\in \mathcal{L}$,
   $$
    \left[\varphi,\psi\right](\llbracket x,y,z\rrbracket) = (\varphi\circ\psi - \psi\circ\varphi)( \llbracket x,y,z\rrbracket).
   $$
    Using the centroid property of $\varphi$ and $\psi$, one can check that
   $$
    \llbracket\left[\varphi,\psi\right](x),y,z\rrbracket = \llbracket\llbracket x,\left[\varphi,\psi\right](y),z\rrbracket
		=\llbracket x,y,\left[\varphi,\psi\right](z)\rrbracket = 0,
   $$
    which implies $\left[\varphi,\psi\right]$ acts trivially on triple products, but $\left[\varphi,\psi\right](\mathcal{L}) \subseteq \Gamma(\mathcal{L})$.
\end{enumerate}
\end{proof}

\begin{prop}
Let $(\mathcal{L},\llbracket-,-,-\rrbracket)$ be a ternary Leibniz algebra.
For any $\kappa\in \mathrm{Der}(\mathcal{L})$ and any $\varphi\in \Gamma(\mathcal{L})$, the following statements hold:
\begin{enumerate}
\item[(i)] The composition $\kappa\circ\varphi$ belongs to $\Gamma(\mathcal{L})$ if and only if
$\varphi\circ\kappa$ is a central derivation of $\mathcal{L}$.
\item[(ii)] The composition $\kappa\circ\varphi$ is a derivation of $\mathcal{L}$ if and only if
$[\kappa,\varphi]=\kappa\circ\varphi-\varphi\circ\kappa$ is a central derivation of $\mathcal{L}$.
\end{enumerate}
\end{prop}

\begin{proof}
\begin{enumerate}
\item[(i)]
Let $\kappa\in \mathrm{Der}(\mathcal{L})$ and $\varphi\in \Gamma(\mathcal{L})$.
For all $x,y,z\in\mathcal{L}$, since $\varphi$ belongs to the centroid, we have
$$
\varphi(\llbracket x,y,z\rrbracket)=\llbracket\varphi(x),y,z\rrbracket.
$$
Applying $\kappa$ and using the derivation property of $\kappa$, we obtain
$$
\kappa\circ\varphi(\llbracket x,y,z\rrbracket)=\llbracket\kappa\circ\varphi(x),y,z\rrbracket+\llbracket\varphi(x),\kappa(y),z\rrbracket+\llbracket\varphi(x),y,\kappa(z)\rrbracket.
$$
On the other hand,
$$
\varphi\circ\kappa(\llbracket x,y,z\rrbracket)=\llbracket \varphi\circ\kappa(x),y,z\rrbracket
+
\llbracket\varphi(x),\kappa(y),z\rrbracket
+
\llbracket \varphi(x),y,\kappa(z)].
$$
Subtracting these two equalities yields
$$
(\kappa\circ\varphi-\varphi\circ\kappa)(\llbracket x,y,z\rrbracket)
=
\llbracket(\kappa\circ\varphi-\varphi\circ\kappa)(x),y,z\rrbracket.
$$
Hence $\kappa\circ\varphi\in\Gamma(\mathcal{L})$ if and only if
$$
(\varphi\circ\kappa)(\llbracket x,y,z\rrbracket)=0,
\quad \forall x,y,z\in\mathcal{L},
$$
which is equivalent to $\varphi\circ\kappa$ being a central derivation.

\item[(ii)]
Assume first that $\kappa\circ\varphi$ is a derivation of $\mathcal{L}$.
Since $\varphi\in\Gamma(\mathcal{L})$, it follows that $\left[\kappa,\varphi\right]\in\Gamma(\mathcal{L})$.
Thus, for all $x,y,z\in\mathcal{L}$,
\begin{equation}\label{eq1}
\left[\kappa,\varphi\right]([x,y,z\rrbracket)
=
\llbracket\left[\kappa,\varphi\right](x),y,z\rrbracket
=
\llbracket x,[\kappa,\varphi](y),z\rrbracket
=
\llbracket x,y,\left[\kappa,\varphi\right](z)].
\end{equation}
Moreover, since both $\kappa\circ\varphi$ and $\varphi\circ\kappa$ are derivations, we have
\begin{equation}\label{eq2}
\begin{aligned}
\left[\kappa,\varphi\right](\llbracket x,y,z\rrbracket)
&=
\llbracket\kappa\circ\varphi(x),y,z\rrbracket
+
\llbracket x,\kappa\circ\varphi(y),z\rrbracket
+
\llbracket x,y,\kappa\circ\varphi(z)\rrbracket \\
&\quad -
\llbracket\varphi\circ\kappa(x),y,z\rrbracket
-
\llbracket x,\varphi\circ\kappa(y),z\rrbracket
-
\llbracket x,y,\varphi\circ\kappa(z)\rrbracket.
\end{aligned}
\end{equation}
Comparing \eqref{eq1} and \eqref{eq2}, we obtain
$$
\llbracket\left[\kappa,\varphi\right](x),y,z\llbracket=\llbracket x,\llbracket\kappa,\varphi\rrbracket(y),z\rrbracket=\llbracket x,y,\llbracket \kappa,\varphi\rrbracket(z)\rrbracket=0,
$$
which shows that $\left[\kappa,\varphi\right]$ is a central derivation.

Conversely, assume that $[\kappa,\varphi]$ is a central derivation.
Then for all $x,y,z\in\mathcal{L}$,
$$
\kappa\circ\varphi(\llbracket x,y,z\rrbracket)=\llbracket\varphi\circ\kappa(x),y,z\rrbracket
+
\llbracket x,\varphi\circ\kappa(y),z\rrbracket
+
\llbracket x,y,\varphi\circ\kappa(z)\rrbracket,
$$
which proves that $\kappa\circ\varphi$ satisfies the derivation identity.
\end{enumerate}
\end{proof}


\begin{prop}\label{pQ}
Let $(\mathcal{L},\llbracket-,-,-\rrbracket)$   be a  finite dimensional ternary Leibniz algebra with basis $\left\{1,2,\cdots,n\right\}$. A linear transformation $\Phi : \mathcal{L}\rightarrow \mathcal{L}$
belongs to $QCent(\mathcal{L})$ if and only if 
$$\llbracket\Phi(e_i), e_j, e_k\rrbracket=\llbracket e_i, \Phi(e_j), e_k\rrbracket=\llbracket e_i, e_j, \Phi(e_k)\rrbracket$$ for all $i, j\in \left\{1,2,\cdots,n\right\}.$
\end{prop}
\begin {proof}
Necessity follows directly from Definition \ref{qc}. For sufficiency, consider arbitrary elements\\ 
$$x=\displaystyle\sum^{n}_{i=1}x_ie_i\in \mathcal{L},\quad y=\sum^{m}_{j=1}y_je_j\in \mathcal{L},\,\text{and}\, z=\sum^{m}_{k=1}z_ke_k\in \mathcal{L}.$$
Assuming the condition holds on basis elements, compute the following: 
$$\begin{array}{ll}
\displaystyle\llbracket\Phi(x), y, z\rrbracket
&=\displaystyle\llbracket\Phi(\sum^{n}_{i=1}x_ie_i),\sum^{m}_{j=1}y_je_j, \sum^{m}_{k=1}z_ke_k\rrbracket\\
&=\displaystyle\llbracket\sum^{n}_{i=1}x_i\Phi(e_i),\sum^{m}_{j=1}y_je_j, \sum^{m}_{k=1}y_ke_k\rrbracket\\
&=\displaystyle\sum^{n}_{i=1}\sum^{n}_{j=1}x_iy_jz_k(\llbracket\Phi(e_i),e_j,e_k\rrbracket)\\
&=\displaystyle\sum^{n}_{i=1}\sum^{n}_{j=1}\sum^{n}_{k=1}x_iy_jz_k(\llbracket e_i, \Phi(e_j), e_k\rrbracket)\\
&=\displaystyle\llbracket\sum^{n}_{i=1}x_ie_i,\sum^{m}_{j=1}y_j\Phi(e_j), \sum_{k=1}^nz_ke_k\rrbracket\\
&=\displaystyle\llbracket x,\Phi(\sum^{m}_{j=1}y_je_j),z\rrbracket=\llbracket x, \Phi(y), z\rrbracket.
\end{array}$$

$$\begin{array}{ll}
\displaystyle\llbracket\Phi(x), y, z\rrbracket
&=\displaystyle\llbracket\Phi(\sum^{n}_{i=1}x_ie_i),\sum^{m}_{j=1}y_je_j, \sum^{m}_{k=1}z_ke_k\rrbracket\\
&=\displaystyle\llbracket\sum^{n}_{i=1}x_i\Phi(e_i),\sum^{m}_{j=1}y_je_j, \sum^{m}_{k=1}y_ke_k\rrbracket\\
&=\displaystyle\sum^{n}_{i=1}\sum^{n}_{j=1}x_iy_jz_k(\llbracket\Phi(e_i),e_j,e_k\rrbracket)\\
&=\displaystyle\sum^{n}_{i=1}\sum^{n}_{j=1}\sum^{n}_{k=1}x_iy_jz_k(\llbracket e_i, e_j,\Phi(e_k)\rrbracket)\\
&=\displaystyle\llbracket\sum^{n}_{i=1}x_ie_i, \sum_{j=1}^nz_je_j,\sum^{m}_{k=1}z_k\Phi(e_k)\rrbracket\\
&=\displaystyle\llbracket x,\Phi(\sum^{m}_{j=1}y_je_j),z\rrbracket=\llbracket x,y, \Phi(z)\rrbracket.
\end{array}$$

$$\begin{array}{ll}
\displaystyle\llbracket x, \Phi(y), z\rrbracket
&=\displaystyle\llbracket\sum^{m}_{i=1}x_ie_i, \Phi(\sum^{n}_{j=1}y_je_j), \sum^{m}_{k=1}z_ke_k\rrbracket\\
&=\displaystyle\llbracket ,\sum^{m}_{i=1}x_ie_i, \sum^{n}_{j=1}x_j\Phi(e_j), \sum^{m}_{k=1}y_ke_k\rrbracket\\
&=\displaystyle\sum^{n}_{i=1}\sum^{n}_{j=1}\sum^{n}_{k=1}x_iy_jz_k(\llbracket e_i,\Phi(e_j), e_k\rrbracket)\\
&=\displaystyle\sum^{n}_{i=1}\sum^{n}_{j=1}\sum^{n}_{k=1}x_iy_jz_k(\llbracket e_i, e_j,\Phi(e_k)\rrbracket)\\
&=\displaystyle\llbracket\sum^{n}_{i=1}x_ie_i, \sum_{j=1}^ny_je_j,\sum^{m}_{k=1}z_k\Phi(e_k)\rrbracket\\
&=\displaystyle\llbracket x,\Phi(\sum^{m}_{j=1}y_je_j),z\rrbracket=\llbracket x,y, \Phi(z)\rrbracket.
\end{array}$$

This establishes the sufficiency condition.
\end {proof}
For computation purposes, we fix a basis $\left\{e_1,e_2,\cdots,e_n\right\}$ of $L$ and represent linear transformations using matrices.  If $\Phi: \mathcal{L}\rightarrow \mathcal{L}$ has a matrix 
representation $\Phi=(a_{ij})$ relative to this basis, that is, $$\Phi(e_i)=\sum_{j=1}^na_{ji}e_{j},\quad \text{for}\, i=1,\ldots,n.$$

Let $(\mathcal{L}, \llbracket-,-,-\rrbracket)$ be an $n$-dimensional Ternary Leibniz algebra, $\{e_i\}$ be a basis of $\mathcal{L}$ and let $\Phi$ be a quasi-centroid on $\mathcal{L}$.
 For any $i, j\in \mathbb{N}, 1\leq i, j\leq n$, let us put 
$$\llbracket e_i, e_j, e_k\rrbracket=\sum_{p=1}^{n}\chi_{ijk}^pe_p, \quad\Phi(e_i)=\sum_{j=1}^na_{ji}e_j.$$
The axioms in Definition \ref{dt} are, respectively, equivalent to
$$
a_{pi}\chi_{pjk}^q=a_{pj}\chi_{ipk}^q=a_{pk}\chi_{ijp}^q.
$$

\begin{thm}
The matrix representation of the quasi-centroid of the four-dimensional non-abelian ternary Leibniz algebra
$\mathcal{L}_4^{1}$ with respect to the basis $\{e_1,e_2,e_3,e_4\}$ is given by
$$
QCent(\mathcal{L}_4^1)
=
\left\{\left(\begin{array}{ccccc}
a_{11}&0&0&0\\
a_{21}&a_{22}&a_{23}&0\\
0&0&a_{11}&a_{24}\\
a_{41}&a_{42}&a_{43}&a_{44}
\end{array}
\right), a_{ij}\in\mathbb{K}\right\}.
$$
Moreover, a basis of $QCent(\mathcal{L}_4^1)$ is given by
$$
\{e_{11},e_{21},e_{41},e_{22},e_{42},e_{23},e_{43},e_{24},e_{44}\}.
$$
\end{thm}

\begin{proof}
By Proposition~\ref{pQ}, a linear map $\Phi\in End_{\mathbb{K}}(\mathcal{L}_4^1)$ belongs to
$QCent(\mathcal{L}_4^1)$ if and only if
$$\llbracket \Phi(x),y,z\rrbracket=\llbracket x,\Phi(y),z\rrbracket=\llbracket x,y,\Phi(z)\rrbracket
\quad\text{for all } x,y,z\in\mathcal{L}_4^1.
$$
Let
$$\Phi=
\left(\begin{array}{ccccc}
a_{11}&a_{12}&a_{13}&a_{14}\\
a_{21}&a_{22}&a_{23}&a_{24}\\
a_{31}&a_{32}&a_{33}&a_{34}\\
a_{41}&a_{42}&a_{43}&a_{44}
\end{array}
\right),$$
be the matrix of $\Phi$ relative to the basis $\{e_1,e_2,e_3,e_4\}$.
Recall that the nonzero ternary products of $\mathcal{L}_4^1$ are
$$\llbracket e_1,e_1,e_1\rrbracket=\llbracket e_1,e_1,e_3\rrbracket=\llbracket e_1,e_3,e_1\rrbracket
=\llbracket e_3,e_1,e_1\rrbracket=\llbracket e_3,e_1,e_3\rrbracket=\llbracket e_3,e_3,e_3\rrbracket=e_2+e_4.
$$
We now impose the quasi-centroid conditions.

\begin{itemize}
\item \textbf{From the identities}
$$
\llbracket\Phi(e_1),e_1,e_1\rrbracket=\llbracket e_1,\Phi(e_1),e_1\rrbracket=\llbracket e_1,e_1,\Phi(e_1)\rrbracket,
$$
we obtain
$$
a_{31}=0.
$$
\item\textbf{From}
$$
\llbracket e_1,e_1,\Phi(e_2)\rrbracket=\llbracket\Phi(e_1),e_1,e_2\rrbracket=\llbracket e_1,\Phi(e_1),e_2\rrbracket,
$$
we deduce
$$
a_{32}=0.
$$
\item \textbf{Using the relations involving} $\Phi(e_3)$, namely
$$
\llbracket\Phi(e_1),e_1,e_3\rrbracket=\llbracket e_1,e_1,\Phi(e_3)\rrbracket,
$$
we obtain
$$
a_{11}=a_{33}, \qquad a_{13}=0.
$$
\item \textbf{From the conditions involving} $\Phi(e_4)$,
$$
\llbracket e_1,e_1,\Phi(e_4)\rrbracket=\llbracket\Phi(e_1),e_1,e_4\rrbracket,
$$
and
$$
\llbracket e_1,e_3,\Phi(e_4)\rrbracket=\llbracket\Phi(e_1),e_3,e_4\rrbracket,
$$
we obtain
$$
a_{14}=a_{34}=0.
$$
\item \textbf{Finally, from}
$$
\llbracket \Phi(e_2),e_1,e_1\rrbracket=\llbracket e_2,\Phi(e_1),e_1\rrbracket,
$$
we deduce
$$
a_{12}=0.
$$
\end{itemize}
\medskip
\noindent
By collecting all the obtained constraints, the matrix representation of $\Phi$ is of the following form:

$$
QCent(\mathcal{L}_4^1)
=
\left\{\left(\begin{array}{ccccc}
a_{11}&0&0&0\\
a_{21}&a_{22}&a_{23}&0\\
0&0&a_{11}&a_{24}\\
a_{41}&a_{42}&a_{43}&a_{44}
\end{array}
\right), a_{ij}\in\mathbb{K}\right\}.
$$
which proves the claimed description of $QCent(\mathcal{L}_4^1)$ and therefore the result follows.
\end{proof}

The following table summarizes the explicit matrix descriptions and dimensions of centroids $Cent(\mathcal{L})$ and quasi-centroids
$QCent(\mathcal{L})$ for all four-dimensional non-abelian ternary Leibniz algebras $\mathcal{L}$.

\begin{longtable}{||c||c||c||c||c||c||c||c||c||c||c||c||}
\caption{Four-dimentionnal cas}
\\ \hline
\textbf{$\mathcal{L}$}&\textbf{$Cent(\mathcal{L})$}&\textbf{Dim}&\textbf{$QCent(\mathcal{L})$}&\textbf{Dim}
\\ \hline
$\mathcal{L}_4^{1}$&
 $\left(\begin{array}{cccc}
a_{11}&0&0&0\\
0&a_{11}&0&0\\
0&0&a_{11}&0\\
0&0&0&a_{11}
\end{array}
\right)$
&
1
&
$\left(\begin{array}{ccccc}
a_{11}&0&0&0\\
0&a_{11}&0&0\\
a_{31}&a_{32}&a_{11}&a_{34}\\
a_{41}&a_{42}&a_{43}&a_{11}
\end{array}
\right)$
&
9
\\ \hline 
$\mathcal{L}_4^{2}$&
 $\left(\begin{array}{cccc}
a_{11}&0&0&0\\
0&a_{11}&0&0\\
a_{31}&a_{32}&c_{11}&0\\
-a_{31}&-a_{32}&0&a_{11}
\end{array}
\right)$
&
3
&
$\left(\begin{array}{ccccc}
a_{11}&0&0&0\\
a_{21}&a_{22}&a_{23}&a_{24}\\
0&0&a_{11}&0\\
a_{41}&a_{42}&a_{43}&a_{44}
\end{array}
\right)$
&
9
\\ \hline 
$\mathcal{L}_4^{3}$&
 $\left(\begin{array}{cccc}
a_{11}&0&0&0\\
0&a_{11}&0&0\\
0&0&a_{11}&0\\
0&0&0&a_{11}
\end{array}
\right)$
&
1
&
$\left(\begin{array}{ccccc}
a_{11}&0&0&0\\
0&a_{11}&0&0\\
0&0&a_{11}&0\\
0&0&0&a_{11}
\end{array}
\right)$
&
9
\\ \hline
$\mathcal{L}_4^{4}$&
 $\left(\begin{array}{cccc}
a_{11}&0&0&0\\
0&a_{11}&0&0\\
0&0&a_{11}&0\\
0&0&0&a_{11}
\end{array}
\right)$
&
1
&
$\left(\begin{array}{ccccc}
a_{11}&0&0&0\\
0&a_{11}&0&0\\
a_{31}&a_{32}&a_{33}&a_{34}\\
0&0&0&a_{11}
\end{array}
\right)$
&
5
\\ \hline 
$\mathcal{L}_4^{5}$&
 $\left(\begin{array}{cccc}
a_{11}&0&0&0\\
0&a_{11}&0&0\\
0&0&a_{11}&0\\
0&0&0&a_{11}
\end{array}
\right)$
&
1
&
$\left(\begin{array}{ccccc}
a_{11}&0&0&0\\
0&a_{11}&0&0\\
0&0&a_{11}&0\\
a_{41}&a_{42}&a_{43}&a_{44}
\end{array}
\right)$
&
5
\\ \hline 
$\mathcal{L}_4^{6}$&
 $\left(\begin{array}{cccc}
a_{11}&0&0&0\\
0&a_{11}&0&0\\
0&0&a_{11}&0\\
0&0&0&a_{11}
\end{array}
\right)$
&
1
&
$\left(\begin{array}{ccccc}
a_{11}&0&0&0\\
0&a_{11}&0&0\\
a_{31}&a_{32}&a_{33}&a_{34}\\
0&0&0&a_{11}
\end{array}
\right)$
&
5
\\ \hline 
$\mathcal{L}_4^{7\alpha}$&
 $\left(\begin{array}{cccc}
a_{44}&0&0&0\\
0&a_{44}&0&0\\
0&0&a_{44}&0\\
0&0&0&a_{44}
\end{array}
\right)$
&
1
&
$\left(\begin{array}{ccccc}
a_{22}&0&0&0\\
0&a_{22}&0&0\\
0&0&a_{22}&0\\
0&0&0&a_{22}
\end{array}
\right)$
&
1
\\ \hline 
$\mathcal{L}_4^{8}$&
 $\left(\begin{array}{cccc}
a_{11}&0&0&0\\
0&a_{11}&0&0\\
0&0&a_{11}&0\\
0&0&0&a_{11}
\end{array}
\right)$
&
1
&
$\left(\begin{array}{ccccc}
a_{11}&0&0&0\\
a_{21}&a_{22}&a_{23}&a_{24}\\
0&0&a_{11}&0\\
0&0&0&a_{11}
\end{array}
\right)$
&
5
\\ \hline 
$\mathcal{L}_4^{9}$&
 $\left(\begin{array}{cccc}
a_{11}&0&0&0\\
0&a_{11}&0&0\\
0&0&a_{11}&0\\
0&0&0&a_{11}
\end{array}
\right)$
&
1
&
$\left(\begin{array}{ccccc}
a_{11}&0&0&0\\
a_{21}&a_{22}&a_{23}&a_{24}\\
0&0&a_{11}&0\\
0&0&0&a_{11}
\end{array}
\right)$
&
5
\\ \hline 
$\mathcal{L}_4^{10}$&
 $\left(\begin{array}{cccc}
a_{11}&0&0&0\\
0&a_{11}&0&0\\
0&0&a_{11}&0\\
0&0&0&a_{11}
\end{array}
\right)$
&
1
&
$\left(\begin{array}{ccccc}
a_{11}&0&0&0\\
0&a_{11}&0&0\\
0&0&a_{11}&0\\
0&0&0&a_{11}
\end{array}
\right)$
&
1
\\ \hline 
$\mathcal{L}_4^{11}$&
 $\left(\begin{array}{cccc}
a_{11}&0&0&0\\
0&a_{11}&0&0\\
0&0&a_{11}&0\\
0&0&0&a_{11}
\end{array}
\right)$
&
1
&
$\left(\begin{array}{ccccc}
a_{11}&0&0&0\\
0&a_{11}&0&0\\
0&0&a_{33}&0\\
a_{41}&a_{42}&a_{43}&a_{44}
\end{array}
\right)$
&
6
\\ \hline
$\mathcal{L}_4^{12}$&
 $\left(\begin{array}{cccc}
a_{11}&0&0&0\\
0&a_{11}&0&0\\
0&0&a_{11}&0\\
0&0&0&a_{11}
\end{array}
\right)$
&
1
&
$\left(\begin{array}{ccccc}
a_{11}&0&0&0\\
0&a_{11}&0&0\\
0&0&a_{33}&0\\
0&0&0&a_{11}
\end{array}
\right)$
&
2
\\ \hline 
$\mathcal{L}_4^{13}$&
 $\left(\begin{array}{cccc}
a_{11}&0&0&0\\
0&a_{11}&0&0\\
0&0&a_{11}&0\\
0&0&0&a_{11}
\end{array}
\right)$
&
1
&
$\left(\begin{array}{ccccc}
a_{11}&0&0&0\\
0&a_{11}&0&0\\
0&0&a_{11}&0\\
0&0&0&a_{11}
\end{array}
\right)$
&
1
\\ \hline 
$\mathcal{L}_4^{14}$&
 $\left(\begin{array}{cccc}
a_{11}&0&0&0\\
0&a_{11}&0&0\\
0&0&a_{11}&0\\
0&0&0&a_{11}
\end{array}
\right)$
&
1
&
$\left(\begin{array}{ccccc}
a_{11}&0&0&0\\
0&a_{11}&0&0\\
0&0&a_{11}&0\\
a_{41}&a_{42}&a_{43}&a_{44}
\end{array}
\right)$
&
5
\\ \hline 
$\mathcal{L}_4^{15}$&
 $\left(\begin{array}{cccc}
a_{11}&0&0&0\\
0&a_{11}&0&0\\
0&0&a_{11}&0\\
0&0&0&a_{11}
\end{array}
\right)$
&
1
&
$\left(\begin{array}{ccccc}
a_{11}&0&0&0\\
0&a_{11}&0&0\\
0&0&a_{11}&0\\
0&0&0&a_{11}
\end{array}
\right)$
&
1
\\ \hline 
$\mathcal{L}_4^{16}$&
 $\left(\begin{array}{cccc}
a_{11}&c_{12}&a_{13}&0\\
0&a_{22}&0&0\\
0&0&a_{22}&0\\
a_{22}-a_{11}&-a_{12}&-a_{13}&a_{22}
\end{array}
\right)$
&
5
&
$\left(\begin{array}{ccccc}
a_{11}&0&0&0\\
0&a_{11}&0&0\\
0&0&a_{11}&0\\
a_{41}&a_{42}&a_{43}&a_{44}
\end{array}
\right)$
&
5
\\ \hline 
$\mathcal{L}_4^{17}$&
 $\left(\begin{array}{cccc}
a_{11}&c_{12}&a_{13}&0\\
0&a_{22}&0&0\\
0&0&a_{22}&0\\
a_{22}-a_{11}&-a_{12}&-a_{13}&a_{22}
\end{array}
\right)$
&
5
&
$\left(\begin{array}{ccccc}
a_{11}&0&0&0\\
0&a_{11}&0&0\\
0&0&a_{11}&0\\
a_{41}&a_{42}&a_{43}&a_{44}
\end{array}
\right)$
&
5
\\ \hline 
$\mathcal{L}_4^{18}$&
 $\left(\begin{array}{cccc}
a_{11}&0&0&0\\
0&a_{11}&0&0\\
0&0&a_{11}&0\\
0&0&0&a_{11}
\end{array}
\right)$
&
1
&
$\left(\begin{array}{ccccc}
a_{11}&0&0&0\\
0&a_{11}&0&0\\
a_{31}&a_{32}&a_{33}&a_{34}\\
0&0&0&a_{11}
\end{array}
\right)$
&
5
\\ \hline 
\end{longtable}

\newpage
\begin{re}
\begin{itemize}\,
\item $Dim(Cent(\mathcal{L}))\leq Dim(QCent(\mathcal{L})).$
	\item The dimensions of the \textbf{centroids} of $4$-dimensional ternary Leibniz algebras range between 1 and 5.
	\item The dimensions of the \textbf{quasi-centroids} of $4$-dimensional ternary Leibniz algebras range between 1 and 9.
\end{itemize}
\end{re}

\section*{Conclusion}

In this work, we classify all four-dimensional ternary Leibniz algebras and determine explicitly their centroids and quasi-centroids. 
For each non-abelian algebra, these structures are described in matrix form and their dimensions computed, providing insight into the
 internal symmetries of low-dimensional ternary Leibniz algebras. 
Our results highlight the structural richness of these algebras and extend classical results from the binary setting, opening 
perspectives for the study of centroids, quasi-centroids, and related structures in higher-dimensional and $n$-ary algebras.\\

	\noindent {\bf Acknowledgment:}
	The authors would like to thank the referee for valuable comments and suggestions on this article. \\
	
	\noindent {\bf Conflicts of Interest:} The authors declare no conflict of interest.

\end{document}